\documentclass[a4paper,12pt]{article}
\usepackage{}
\usepackage{mathrsfs}
\usepackage{amssymb}
\usepackage{amsmath}
\usepackage{amsmath,amssymb,amsthm,graphics, graphicx,latexsym,amsfonts}
\usepackage{fancyhdr}
\usepackage{color}
\usepackage{bm}
\usepackage{galois}
\usepackage{tikz}
\usetikzlibrary{calc}
\newcommand{\RNum}[1]{\lowercase\expandafter{\romannumeral #1\relax}}

\title{A generalization of Noel-Reed-Wu Theorem to signed graphs\thanks{Supported by the
National Natural Science Foundation of China under Grant No.\,11471273 and 11561058.}}
\author
{
Wei Wang$^{\rm a,b}$,
Jianguo Qian$^{\rm a}$\thanks{Corresponding author: jgqian@xmu.edu.cn.}
\\
{\footnotesize$^{\rm a}$School of Mathematical Sciences, Xiamen University, Xiamen 361005, P. R. China}\\
{\footnotesize$^{\rm b}$College of Information Engineering, Tarim University, Alar 843300, P. R. China}
}
\usepackage{indentfirst}
\date{}
\begin{document}
\maketitle
\newtheorem{lem}{Lemma}[section]
\newtheorem{thm}[lem]{Theorem}
\newtheorem{prop}[lem]{Proposition}
\newtheorem{cor}[lem]{Corollary}
\newtheorem{conj}[lem]{Conjecture}
\newtheorem{defi}[lem]{Definition}
\newtheorem{prob}[lem]{Problem}
\newtheorem*{pf}{Proof}
\begin{abstract}
Let $\Sigma$ be a  signed graph where two edges joining the same pair of vertices with opposite signs are allowed.  The zero-free chromatic number $\chi^*(\Sigma)$ of $\Sigma$ is the minimum even integer $2k$ such that $G$ admits a proper coloring $f\colon\,V(\Sigma)\mapsto \{\pm 1,\pm 2,\ldots,\pm k\}$. The zero-free list chromatic number $\chi^*_l(\Sigma)$ is the list version of zero-free chromatic number.  $\Sigma$ is called zero-free chromatic-choosable if $\chi^*_l(\Sigma)=\chi^*(\Sigma)$.  We show that if $\Sigma$ has at most $\chi^*(\Sigma)+1$ vertices then $\Sigma$ is zero-free chromatic-choosable. This result strengthens Noel-Reed-Wu Theorem which states that every graph $G$ with at most $2\chi(G)+1$ vertices is chromatic-choosable, where $\chi(G)$ is the chromatic number of $G$.
\end{abstract}

\section{Introduction}
A graph is called \emph{chromatic-choosable} \cite{Ohba2002} if its list chromatic number equals its chromatic number. Characterizing which graphs are chromatic-choosable is a challenging problem  in the field of list coloring. A recent breakthrough is the following theorem of Noel-Reed-Wu \cite{Noel2015}, which was conjectured by Ohba \cite{Ohba2002} in 2002.
\begin{thm}\cite{Noel2015}\label{NRWthm}
If $|V(G)|\le 2\chi(G)+1$, then $G$ is chromatic-choosable.
\end{thm}
The main purpose of this paper is to extend the above theorem to signed graphs. To state our result we need some definitions.

A signed graph $\Sigma$ is a pair $(G,\sigma)$, where $G$ is a loopless graph and $\sigma$ is a mapping from $E(G)$ to $\{+1,-1\}$. An edge $e$ is \emph{positive} (resp. \emph{negative}) if $\sigma(e)=+1$ (resp. $\sigma(e)=-1$).  Throughout this paper, two edges joining the same pair of vertices with opposite signs are allowed but that with the same signs are not allowed.

Let  $\mathbb{Z}^*$ be the set of nonzero integers, i.e.,  $\mathbb{Z}^*=\mathbb{Z}\setminus \{0\}$.  For $C\subseteq Z^*$, a \emph{proper coloring} \cite{Zaslavsky1982} of $\Sigma$ with color set $C$ is a mapping  $f\colon\,V(\Sigma)\mapsto C$ such that for each edge $e$,
\begin{equation}
f(u)\neq \sigma(e)f(v) ~\text{if $u$ and $v$ are joined by $e$}.
\end{equation}
In particular, if $u$ and $v$ are joined by two edges with opposite signs  then $f(u)$ and $f(v)$ have different absolute values.

A {\it zero-free $2k$-coloring} \cite{Zaslavsky1982} of a signed graph $\Sigma$ is a proper coloring of $\Sigma$ with color set $\{\pm 1,\pm 2,\ldots,\pm k\}$.  A signed graph is {\it zero-free $2k$-colorable} if it admits a zero-free $2k$-coloring. The \emph{zero-free chromatic number} of a signed graph $\Sigma$, denoted $\chi^*(\Sigma)$, is the minimum even integer $2k$ for which $\Sigma$ is zero-free $2k$-colorable.

For a signed graph $\Sigma$, a zero-free list assignment is a mapping $L$ which assigns each vertex $v$ a set $L(v)$ of permissible colors in $\mathbb{Z}^*$. For a zero-free list assignment $L$ of $\Sigma$, an {\it $L$-coloring} is a proper coloring $f$ such that $f(v)\in L(v)$ for all $v\in V(\Sigma)$. We say that $\Sigma$ is {\it $L$-colorable} if $\Sigma$ admits an $L$-coloring. The \emph{zero-free list chromatic number} of $\Sigma$, denoted $\chi_l^*(\Sigma)$, is the minimum $2k$ such that $\Sigma$ is $L$-colorable for any zero-free list assignment $L$ with $|L(v)|\ge 2k$ for all $v\in V(\Sigma)$. Clearly, $\chi_l^*(\Sigma)\ge \chi^*(\Sigma)$ for any signed graph $\Sigma$. A signed graph $\Sigma$ is \emph{zero-free chromatic-choosable} if $\chi_l^*(\Sigma)=\chi^*(\Sigma)$. Under these definitions, the main result of this paper  is stated as follows:
\begin{thm}\label{main}
If $|V(\Sigma)|\le \chi^*(\Sigma)+1$, then $\Sigma$ is zero-free chromatic-choosable.
\end{thm}

For a  graph $G$, the \emph{complete expansion} of $G$ is a signed graph obtained from $G$ by regarding each edge in $G$ as a positive edge and adding a negative edge between each pair of vertices.

\begin{lem}\label{basicrelation}
Let $G$ be a graph and $\Sigma$ be its complete expansion. Then $\chi^*(\Sigma)=2\chi(G)$ and $\chi_l^*(\Sigma)\ge 2\chi_l(G)$.
\end{lem}
\begin{proof}
Let  $f\colon\,V(G)\mapsto \{1,2,\ldots,k\}$ be a $k$-coloring of $G$. Note that $V(G)=V(\Sigma)$.  Consider the mapping $\tilde{f}\colon\,V(\Sigma)\mapsto \{\pm 1,\pm 2,\ldots,\pm k\}$ defined by $\tilde{f}(v)=f(v)$ for each $v\in V(\Sigma)$. One easily verify that $\tilde{f}$ is a zero-free $2k$-coloring of $\Sigma$. By letting $k=\chi(G)$, we find that $\Sigma$ is $2\chi(G)$-colorable and hence $\chi^*(\Sigma)\le 2\chi(G)$.

Similarly, let $g\colon\,V(\Sigma)\mapsto \{\pm 1,\pm 2,\ldots,\pm 2k\}$ be a zero-free $2k$-coloring of $\Sigma$. Then the mapping $\bar{g}\colon\,V(G)\mapsto \{1,2,\ldots,k\}$, defined by $\bar{g}(v)=|g(v)|$ for each $v\in V(G)$, is a $k$-coloring of $G$. Thus we have $\chi(G)\le \chi^*(\Sigma)/2$. This proves that
 $\chi^*(\Sigma)=2\chi(G)$.

 Let $k=\chi_l^*(\Sigma)/2$ and $L$ be any list assignment of $G$ with $|L(v)|\ge k$ for all $v\in V(G)$. To show that $\chi_l^*(\Sigma)\ge 2\chi_l(G)$, i.e., $\chi_l(G)\le k$, it suffices to show that $G$ admits an $L$-coloring.  Without loss of generality, we may assume that $L(v)\subseteq \mathbb{Z}^+$ for any $v\in V(G)$, where $\mathbb{Z}^+$ is the set of positive integers. Let $\tilde{L}$ be the mapping  defined by $\tilde{L}(v)=L(v)\cup (-L(v))$. Clearly, $\tilde{L}$ is a zero-free list assignment and  for each $v\in V(\Sigma)$, $|\tilde{L}(v)|\ge 2k$, i.e., $|\tilde{L}(v)|\ge \chi_l^*(\Sigma)$. Thus, $\Sigma$ is $\tilde{L}$-colorable. Let $h$ be an $\tilde{L}$-coloring of $\Sigma$. Then the mapping $\bar{h}$ defined by $\bar{h}(v)=|h(v)|$ is clearly an $L$-coloring of $G$. This proves the lemma.
\end{proof}
Now we show that Theorem \ref{main} is  indeed a strengthening of Theorem \ref{NRWthm}.
\begin{cor}
Theorem \ref{main} implies Theorem \ref{NRWthm}.
\end{cor}
\begin{proof}
Let $G$ be any graph with $|V(G)|\le 2\chi(G)+1$ and $\Sigma$ be its complete expansion. By Lemma \ref{basicrelation}, $\chi^*(\Sigma)=2\chi(G)$ and hence
$|V(\Sigma)|\le \chi^*(\Sigma)+1$. Thus, by Theorem \ref{main}, $\chi_l^*(\Sigma)=\chi^*(\Sigma)$. Using Lemma \ref{basicrelation} again, we have
$\chi_l(G)\le \chi(G)$ and hence $\chi_l(G)= \chi(G)$.
\end{proof}

For a signed graph $\Sigma$ and a vertex $v\in V(\Sigma)$, a {\it switching} at $x$ means changing the sign of each edge incident with $v$.  Generally, a switching at a vertex subset $X \subseteq V (\Sigma)$ means switching at every vertex in $X$ one by one. Equivalently, a switching at
$X$ means changing the sign of every edge with exactly one end in $X$. Two signed graphs $\Sigma$ and $\Sigma'$ with the same underling graph are {\it switching equivalent} if $\Sigma'$ can be obtained from $\Sigma$ by a switching at $X$ for some $X\subseteq V(\Sigma)$. It is not difficult to verify that two switching equivalent signed graphs have the same zero-free chromatic number as well as the same zero-free list chromatic number. 
\begin{lem}\label{addswitch}
For any signed graph $\Sigma$, there are a signed graph $\Sigma'$ switching equivalent to $\Sigma$ and a complete  $\chi^*(\Sigma)/2$-partite graph $G$ such that the complete expansion of $G$ can be obtained from $\Sigma'$ by adding some signed edges.
\end{lem}
\begin{proof}
Let $k=\chi^*(\Sigma)/2$ and $f\colon\,V(\Sigma)\mapsto\{\pm 1,\pm 2,\ldots,\pm k\}$ be a zero-free $2k$-coloring of $\Sigma$. For $i\in\{1,2,\ldots,k\}$ define $V_i^+=\{v\in V(\Sigma)\colon\,f(v)=i\}$, $V_i^-=\{v\in V(\Sigma)\colon\,f(v)=-i\}$ and $V_i=V_i^+\cup V_i^-$ for $i\in\{1,2,\ldots,k\}$. Let $u,v\in V_i$ and $e$ be an edge joining $u$ and $v$. It is easy to see that $e$ is positive if and only if $e$ has exactly one end in $V_i^+$.  Let $S=\cup_{i=1}^k V_i^+$ and $\Sigma'$ be obtained from $\Sigma$ by a switching at $S$. Now, for each $i\in \{1,2,\ldots,k\}$, each edge in $\Sigma'$ joining two vertices in $V_i$ is negative. Next by adding as many as possible signed edges between vertices in different classes $V_i$ and $V_j$ for each pair $\{i,j\}\subseteq\{1,2,\ldots,k\}$, we obtain a signed graph $\Sigma''$ in which any two vertices belong to different classes are joined by two edges with opposite signs. Finally, let $G$ be the complete $k$-partite graph with parts $V_1,V_2,\ldots,V_k$. We see that $\Sigma''$ is exactly the complete expansion of $G$. This proves the lemma.
\end{proof}

For a signed graph $\Sigma$, let $\Sigma'$ be the complete expansion of a complete $\chi^*(\Sigma)/2$-partite graph $G$ as defined in Lemma \ref{addswitch}. By Lemma \ref{basicrelation}, we have $\chi^*(\Sigma')=\chi^*(\Sigma)$. As the operation of switching at a vertex and adding an edge does not decrease the zero-free list chromatic number, we have $\chi_l^*(\Sigma')\ge \chi_l^*(\Sigma)$. It follows that if $\Sigma'$ is zero-free chromatic-choosable then so is $\Sigma$. Thus,  to  prove Theorem \ref{main}, it suffices to prove the following theorem.
\begin{thm}\label{equmain}
If $\Sigma$ is a complete expansion of a complete $k$-partite graph $G$ on at most $2k+1$ vertices, then $\chi_l^*(\Sigma)=2k$.
\end{thm}
\section{Proof of Theorem \ref{equmain}}
We follow the method in \cite{Noel2015}. Suppose to the contrary that Theorem \ref{equmain} is false.  Let $\Sigma$ be a counterexample to Theorem \ref{equmain}, that is, $\Sigma$ is a complete expansion of a complete $k$-partite graph $G$ on at most $2k+1$ vertices such that $\chi_l^*(\Sigma)>2k$. We suppose further that $|V(\Sigma)|$ is minimal among all counterexamples.

Let $L$ be any zero-free list assignment of $\Sigma$ such that $|L(v)|\ge 2k$ for all $v\in V(\Sigma)$ and $\Sigma$ is not $L$-colorable. We use $\mathcal{L}$ to denote the set of all such list assignments.

\subsection{Properties of $\Sigma$ and $L\in \mathcal{L}$}

 Let $C_L=\cup_{v\in V(\Sigma)}L(v)$ and, for a set $S\subseteq \mathbb{Z}^*$, let $S^{\pm}=S\cup (-S)$.
\begin{lem} \label{fourconds} If there is a nonnegative integer $\ell$ and a nonempty proper subset $A$ of $V(\Sigma)$  such that

\textup{(a).}  $\Sigma[A]$ admits an $L|_A$-coloring $g$,

\textup{(b).}  $|V(\Sigma-A)|\le 2(k-\ell)+1$,

\textup{(c).}  $\chi^*(\Sigma-A)\le 2(k-\ell)$, and

\textup{(d).}  $|L(v)\setminus g(A)^{\pm }|\ge 2(k-\ell)$ for all $v\in V(\Sigma-A),$\\
then $\Sigma$ is $L$-colorable.
\end{lem}
\begin{proof}
Let $\Sigma'=\Sigma-A$ and $L'(v)=L(v)\setminus g(A)^{\pm }$ for $v\in V(\Sigma')$. We claim that $\chi^*_l(\Sigma')\le 2(k-\ell)$. Note that for any signed graph $\Gamma$, we have $\chi^*_l(\Gamma)\le 2|V(\Gamma)|$. Thus, the claim clearly follows when $|V(\Sigma')|<(k-\ell)$. Now, assume  that $|V(\Sigma')|\ge(k-\ell)$. By (c), we have $\chi^*(\Sigma')\le 2(k-\ell)$. If the inequality is strict then we may obtain a new signed graph $\Sigma''$ from $\Sigma'$ by adding some edges such that $\chi^*(\Sigma'')=2(k-\ell)$ and $\Sigma''$ is the complete expansion of a complete $(k-\ell)$-partite graph. Otherwise let $\Sigma''=\Sigma'$. Either case implies that
\begin{equation}\label{inequequ}
\chi^*(\Sigma'')=2(k-\ell).
\end{equation} By (b) and (\ref{inequequ}), we have $|V(\Sigma'')|\le \chi^*(\Sigma'')+1$. It follows from the minimality of $\Sigma$ that $\chi^*_l(\Sigma'')=2(k-\ell)$. From the construction of $\Sigma''$, we have $\chi^*_l(\Sigma')\le \chi^*_l(\Sigma'')$ and hence $\chi^*_l(\Sigma')\le 2(k-\ell)$. This proves the above claim.

Finally, as $\chi^*_l(\Sigma')\le 2(k-\ell)$,  $\Sigma'$ has an $L'$-coloring $h$ by (d). Note that for any $v\in A$ and $v'\in V(\Sigma')$ we have $|g(v)|\neq |h(v')|$. Combining $g$ with $h$ we obtain an $L$-coloring of $\Sigma$.
\end{proof}
A set $P\subseteq V(\Sigma)$ is called a \emph{part} of $\Sigma$ if $P$ is a partition part of the complete $k$-partite graph $G$. A vertex $v$ is a {\it singleton}  if  $\{v\}$ is a part of $\Sigma$. Throughout the following, we use $\xi$ to denote the number of singletons in $\Sigma$.
\begin{cor}\label{nonsingledisjoint}
If $P$ is a nonsingleton part of $\Sigma$ then $\cap_{v\in P}L(v)=\emptyset$.
\end{cor}
\begin{proof}
Suppose to the contrary that $\cap_{v\in P}L(v)\neq\emptyset$. Let $c\in \cap_{v\in P}L(v)$ and $g(v)=c$ for all $v\in P$. By letting $A=P$ and $\ell=1$ one easily check that all conditions in Lemma \ref{fourconds} are satisfied, which implies that $\Sigma$ is $L$-colorable. This is a contradiction.
\end{proof}
\begin{defi}\textup{
Let $B$ be a bipartite graph with bipartition $(X,Y)$ where $Y\subseteq \mathbb{Z}^*$. A matching $M=\{x_1y_1,x_2y_2,\ldots,x_sy_s\}$ of $B$ is \emph{good} if $y_1,y_2,\ldots,y_s$ have different absolute values, where $x_i\in X$, $y_i\in Y$.}

\end{defi}
For a set $Y\subseteq \mathbb{Z}^*$, we use $\textup{abs}(Y)$ to denote the set consisting of the absolute values of all integers in $Y$, i.e., $\textup{abs}(Y)=\{|y|\colon\,y\in Y\}$. A subset $Y_1$ of $Y$ is called a \emph {representative subset} of $Y$ if
$\textup{abs}(Y_1)=\textup{abs}(Y)$ and any two distinct integers in $Y_1$ have distinct absolute values.

\begin{thm}\label{signHall}
Let $B$ be a bipartite graph with bipartition $(X,Y)$ where $Y\subseteq \mathbb{Z}^*$. Then there is a good matching that saturates $X$ if and only if $|\textup{abs}(N_B(S))|\ge |S|$ for every $S\subseteq X$.
\end{thm}
\begin{proof}
Let $\overline{B}$ be the bipartite graph obtained from $B$ by identifying each pair of opposite values in $Y$ and removing the resulting multiple edges.  In other words, $\overline{B}$ has a bipartition $(X,\textup{abs}(Y))$ and each vertex $x\in X$ is joined to all vertices in $\textup{abs}(N(x))$. It is easy to verifies that $B$ has a good matching saturating $X$ if and only if $\overline{B}$ has a matching saturating $X$. Moreover, $|\textup{abs}(N_B(S))|=|N_{\overline{B}}(S)|$ for every $S\subseteq X$. Thus the theorem follows from classical Hall's theorem.
\end{proof}
Let $B_L$ be the bipartite graph with bipartition $(V(\Sigma),C_L)$ where each vertex $v\in V(\Sigma)$ is joined to the colors of $L(v)$.
\begin{prop}\label{repsat}
There are a representative subset $C$ of $C_L$ and a matching in $B_L$ that saturates $C$.
\end{prop}
\begin{proof}
Let $\overline{B}_L$ be the bipartite graph with bipartition $(V(\Sigma),\textup{abs}(C_L))$ obtained from $B_L$  as in the proof of Theorem \ref{signHall}.  Clearly, it suffices to show that $\overline{B}_L$  has a matching that saturates $\textup{abs}(C_L)$.

Suppose to the contrary that $\overline{B}_L$  has no matching that saturates $\textup{abs}(C_L)$. Then, by Hall's Theorem, there is a set $T\subseteq \textup{abs}(C_L)$ such that $|N_{ \overline{B}_L}(T)|<|T|$. We assume further that $T$ is minimal with respect to this property. Note that in $\overline{B}_L$ each vertex in $\textup{abs}(C_L)$ has at least one neighbor in $V(\Sigma)$. Thus, $|T|\ge 2$. Choose $c\in T$ and denote $S=T\setminus\{c\}$ and $A=N_{\overline{B}_L}(S)$. As $T$ is minimal, we have $|N_{\overline{B}_L}(S')|\ge |S'|$ for any subset $S'$ of $S$. Thus, by Hall's Theorem, there is a matching $M$ that saturates $S$.  Consequently, we have
\begin{equation}\label{equinequ}
|A|\ge |S|=|T|-1\ge |N_{ \overline{B}_L}(T)|\ge |A|
\end{equation}
and hence all equalities must hold simultaneously in (\ref{equinequ}). As $|T|\ge 2$, this proves that $|A|=|S|\ge 1$.

 For each $v\in A$, let $f(v)$ be the color matched to $v$ by $M$.  Then, by the definition of $\overline{B}_L$, we see that $f(v)\in \textup{abs}(L(v))$. Thus,  $f(v)$ or $-f(v)$ appears in $L(v)$. Let $g\colon\,A\mapsto \mathbb{Z}^*$ defined by $g(v)=f(v)$ if $f(v)\in L(v)$ and $g(v)=-f(v)$ otherwise. Then, we see that $g(v)\in L(v)$ and $|g(v)|\neq |g(u)|$ for any $u\in A$ with $u\not=v$. Thus, $g$ is an $L_A$-coloring. Since $A=N_{\overline{B}_L}(S)$, every vertex $v\in V(\Sigma)\setminus A$ must have $\textup{abs}(L(v))\cap S=\emptyset$. As $\textup{abs}(g(A))=S$, this implies that $\textup{abs}(L(v))\cap \textup{abs}(g(A))=\emptyset$, i.e. $(L(v))\cap (g(A)^{\pm})=\emptyset$. Let $\ell=0$.  Then $\ell$, $A$ and $g$ satisfy the conditions of Lemma \ref{fourconds}. Thus, $\Sigma$ is $L$-colorable, a contradiction.
\end{proof}
\begin{cor}\label{resrepsat}
There are a representative subset $C$ of $C_L$ and an injective mapping $h\colon\,C\mapsto V(\Sigma)$ such that $c\in L(h(c))$ for all $c\in C$.
\end{cor}
\begin{proof}
This is simply a restatement of Proposition \ref{repsat}.
\end{proof}
\begin{cor}\label{CLSigma}
$|\textup{abs}(C_L)|<|V(\Sigma)|\le 2k+1$.
\end{cor}
\begin{proof}
The second inequality is our assumption on $\Sigma$. We need to show the first inequality.   Let $C$ be a representative subset of $C_L$ and $h$ be an injective $h\colon\,C\mapsto V(\Sigma)$ from  Corollary \ref{repsat}. Since $h$ is injective, we have $|C|\le |V(\Sigma)|$. However, if $|C|= |V(\Sigma)|$ then the mapping $h$ would be a bijection. Since any two distinct elements have distinct absolute values, the inverse of $h$ is clearly an $L$-coloring of $\Sigma$. This contradicts our assumption that $G$ is not $L$-colorable. Thus $|C|\neq |V(\Sigma)|$ and hence $|C|< |V(\Sigma)|$. Note that $|C|=|\textup{abs}(C_L)|$. This proves the corollary.
\end{proof}
\begin{cor}\label{2disjoint}
If there are $u$ and $v$ in $V(\Sigma)$ such that $L(u)\cap L(v)=\emptyset$, then $L(u)\cup L(v)={C_L}$, $|C_L|=4k$ and $|\textup{abs}({C_L})|=2k$.
\end{cor}
\begin{proof}
By our assumption on $L$, we have $|L(u)|,|L(v)|\ge 2k$. Thus, if $L(u)\cap L(v)=\emptyset$, then $|{C_L}|\ge |L(u)|+|L(v)|\ge 4k$ and hence $|\textup{abs}({C_L})|\ge 2k$. On the other hand, by Corollary \ref{CLSigma}, we have $|\textup{abs}(C_L)|\le 2k$. Thus, equalities must hold in all above inequalities. This proves the corollary.
\end{proof}
\begin{cor}\label{exactorder}
$|V(\Sigma)|=2k+1$.
\end{cor}
\begin{proof}
If $\Sigma$ has a part of size 2, then by Corollary \ref{nonsingledisjoint}, $L(u)\cap L(v)=\emptyset$. Thus, by Propositions \ref{2disjoint} and \ref{CLSigma}, $|\textup{abs}(C)|=2k$ and $|V(\Sigma)|=2k+1$.

Now consider the case that $\Sigma$ has a singleton $\{v\}$. Suppose to the contrary that $|V(\Sigma)|\le2k$. Choose $c\in L(v)$ and set $g(v)=c,A=\{v\}, l=1$. Then by Lemma \ref{fourconds}, $\Sigma$ is $L$-colorable. This is a contradiction.

Finally, we consider the case when $\Sigma$ contains neither a part of size 2 nor a singleton.  Then each part has size at least 3 and hence $3k\le |V(\Sigma)|\le 2k+1$. Thus $k=1$ and $|V(\Sigma)|= 2k+1$. This completes the proof of this corollary.
\end{proof}
\begin{prop}\label{threecases}
If $f\colon\,V(\Sigma)\mapsto C_L$ is a proper coloring (not necessarily proper $L$-coloring), then there are a representative subset $C$ of $C_L$ and a proper surjective coloring $g\colon\,V(\Sigma)\mapsto C$ such that for every color $c\in C$, the color class $g^{-1}(c)$ satisfies \\
\textup{(a)} $g(v)\in L(v)$ for all $v\in g^{-1}(c)$, or\\
\textup{(b)} $f(v)=c$ for all $v\in g^{-1}(c)$, or\\
\textup{(c)} $f(v)=-c$ for all $v\in g^{-1}(c)$.
\end{prop}
\begin{proof}
Let $C$ and $h$ be defined as in Corollary \ref{resrepsat}. For a proper coloring
$g\colon\,V(\Sigma)\mapsto C$ and a color $c\in C$, we say that $g$ agrees with $h$ at $c$ if $g(h(c))=c$.

Let $\tau$ be the projection from $C_L$ to $C$, that is, for each $c\in C_L$, $\tau(c)=c$ if $c\in C$ and $\tau(c)=-c$ otherwise. Note that any two colors in
$f(V(\Sigma))$ have different absolute values. The map $g_0=\tau\comp f$ is clearly a proper coloring of $\Sigma$. Now, let $g\colon\,V(\Sigma)\mapsto C$ be a proper coloring such that for every color $c\in g(V(\Sigma))$, the color class $g^{-1}(c)$ satisfies at least one of (a), (b) and (c). Note that such a coloring exists as either (b) or (c) holds for $g_0$ at any $c\in g_0(V(\Sigma))$. We assume further that the number of colors $c\in C$ at which $g$ agrees with $h$ is maximized. We show that $g$ is surjective. Otherwise, let $c_1\in C\setminus g(V(\Sigma))$ be arbitrary and define  a coloring $g_1$ as follows:
 \begin{equation}
 g_1(v)=
 \begin{cases}
 c_1,& \text{if~}v=h(c_1),\\
g(v),& \text{otherwise}.
\end{cases}
 \end{equation}
Note that for any $u\in V(\Sigma)\setminus \{v\}$, $g_1(u)=g(u)$ and hence $g_1(u)\not\in\{c_1,-c_1\}$. Thus, $g_1$ is proper as $g$ is proper. Moreover, $g_1$ agrees with $h$ at $c_1$ and at every color at which $g$ agrees with $h$.  We shall show that at least one of (a),(b) and (c) holds for $g_1$, which is a contradiction to our choice of $g$ and hence completes the proof.

Let $c\in g_1(V(\Sigma))$. If $c=c_1$ then $g_1^{-1}(c)=\{h(c_1)\}$  and hence $g_1(h(c_1))=c_1\in L(h(c_1))$ by Corollary \ref{resrepsat}. Thus (a) is satisfied. Next we consider the case that $c\in g_1(V(\Sigma))\setminus\{c_1\}$. Clearly, $g_1^{-1}(c)\subseteq g^{-1}(c)$.  As $g$ satisfies at least one of (a), (b) and (c), then so does $g_1$.
\end{proof}
In fact, Proposition \ref{threecases} also holds if we remove (c) from the proposition. We write it as a corollary.
\begin{cor}\label{twocases}
If $f\colon\,V(\Sigma)\mapsto C_L$ is a proper coloring, then there are a representative subset $C$ of $C_L$ and a proper surjective coloring $g\colon\,V(\Sigma)\mapsto C$ such that for every color $c\in C$, the color class $g^{-1}(c)$ satisfies either \\
\textup{(a)} $g(v)\in L(v)$ for all $v\in g^{-1}(c)$, or\\
\textup{(b)} $f(v)=c$ for all $v\in g^{-1}(c)$.
\end{cor}
\begin{proof}
Let $C$ be a representative subset of $C_L$ and $g\colon\,V(\Sigma)\mapsto C$ be the surjective mapping as in Proposition \ref{threecases}. Let
$S$ be the set of colors in $C$ such that (c) in Proposition \ref{threecases} holds.  Let $C_1=(C\setminus S)\cup (-S)$ and define $g_1\colon\,V(\Sigma)\mapsto \mathbb{Z}^*$ as follows:
 \begin{equation}
 g_1(v)=
 \begin{cases}
 -g(v) & \text{if~}g(v)\in S,\\
g(v) & \text{otherwise}.
\end{cases}
 \end{equation}
 One easily check that $(C_1,g_1)$ satisfy either (a) or (b). Moreover, for any $v\in V(\Sigma)$ we have either $g_1(v)\in L(v)$ or $g_1(v)=f(v)$. Either case implies $g_1(v)\in C_L$. Thus $C_1$ is a subset of $C_L$. Clearly $|C_1|=|C|$ and $C_1$ contains no opposite pairs. This shows that $C_1$ is a representative subset of $C_L$.
\end{proof}

\subsection{No Near $L$-coloring for $L\in \mathcal{L}$}
Define
\begin{equation}
\gamma_L=|V(\Sigma)|-|\textup{abs}(C_L)|
\end{equation}
Note that $\gamma_L>0$ by Corollary \ref{CLSigma}.

\begin{defi}\label{deffre} \textup{A color $c\in C_L$ is
\begin{itemize}
\item{\emph{globally frequent} for $L$ if it appears in the lists of at least $k+1$ vertices of $\Sigma$.}
\item{\emph{frequent among singletons} for $L$ if it appears in the lists of at least $\gamma_L$ singletons of $\Sigma$.}
\item{\emph{frequent} if $c$ is either globally frequent or frequent among singletons.}
\end{itemize}
}
\end{defi}
\begin{defi}\textup{
A proper coloring $f\colon\,V(\Sigma)\mapsto C_L$ is a \emph{weak} $L$-\emph{coloring} if for every vertex $v\in V(\Sigma)$, either\\
 \indent (a) $f(v)\in L(v)$, or\\
  \indent (b) $f^{-1}(f(v))=\{v\}$.\\
In addition,  $f$ is a \emph{near} $L$-coloring if  (b) is replaced by the following stronger requirement:\\
 \indent (b$'$) $f^{-1}(f(v))=\{v\}$ and $f(v)$ is frequent.
}
\end{defi}
\begin{prop}\label{nearCsur}
If there is a weak (resp. near) $L$-coloring $f$, then there are a representative subset $C$ of $C_L$ and a weak (resp. near) surjective $L$-coloring $g\colon\, V(\Sigma)\mapsto C$ such that

\textup{(\RNum{1})} For each $v\in V(\Sigma)$, $f(v)\in L(v)$ implies $g(v)\in L(v)$, and

\textup{(\RNum{2})} For each $v\in V(\Sigma)$, $g(v)\not\in L(v)$ implies $g(v)=f(v)$.
\end{prop}
\begin{proof}
 Let $f\colon\,V(\Sigma)\mapsto C_L$ be a weak $L$-coloring. Note that $f$ is proper. By Corollary \ref{twocases}, there are a representative subset $C$ of $C_L$ and a proper surjective coloring $g\colon\,V(\Sigma)\mapsto C$ such that for every $c\in C$, the color class $g^{-1}(c)$ satisfies either\\
\textup{(a)} $g(v)\in L(v)$ for all $v\in g^{-1}(c)$, or\\
\textup{(b)} $f(v)=c$ for all $v\in g^{-1}(c)$.

We show that $g$ is a weak $L$-coloring.  Let $v$ be any vertex in $\Sigma$, say $v\in g^{-1}(c)$ for some $c\in C$. If $g(v)\in L(v)$ then we are done. Now, assume that $g(v)\not\in L(v)$. Then $c$ satisfies  (b), i.e., $f(v)=c$. Note that $c=g(v)$. Thus, we have $f(v)\not\in L(v)$ and so (\RNum{1}) holds. Since $f$ is a weak $L$-coloring,  $f^{-1}(c)=\{v\}$. By (b),  $g^{-1}(c)\subseteq f^{-1}(c)$, implying $g^{-1}(c)=\{v\}$ and hence $g(v)=f(v)$ and so (\RNum{2}) holds. This proves that $g$ is a weak $L$-coloring, in addition, $g$ is a near $L$-coloring if $f$ is near.
\end{proof}

Suppose that $C$ is a representative subset of $C_L$ and $f$ is a proper coloring of $\Sigma$ which maps surjectively to $C$. Let $V_f=\{f^{-1}(c)\colon\,c\in C\}$ be the set of color classes under $f$. We define $B_L^f$ to be the bipartite graph with bipartition $(V_f,C_L)$ where each color class $f^{-1}(c)$ is joined to the colors of $\cap_{v\in f^{-1}(c)}L(v)$.
\begin{lem}\label{manysingL}
Suppose that $f\colon\,V(\Sigma)\mapsto C$ is a surjective weak $L$-coloring of $\Sigma$, where $C$ is a representative subset of $C_L$. Let $S$ be a set that maximizes $|S|-|\textup{abs}(N_{B_L^f}(S))|$ over all subsets of $V_f$. If $V_f\setminus S$ contains at least $\gamma_L$ singletons of $\Sigma$, then there is an $L$-coloring of $\Sigma$.
\end{lem}
\begin{proof}
Let $\ell$ be the number of color classes of $f$ with more than one element. Define $X$ to be the union of all color classes in $S$ with exactly one element and define $A=V(\Sigma)\setminus X$. We shall construct an $L_A$-coloring $g$ of $\Sigma[A]$ and prove that $\Sigma$ is $L$-colorable by verifying conditions (a)-(d) in Lemma \ref{fourconds}.

We use a simple equality: $|\textup{abs}(Z_1\cup Z_2)|=|\textup{abs}(Z_1)|+|\textup{abs}(Z_2\setminus Z_1^{\pm})|$ for any finite $Z_1,Z_2\subseteq \mathbb{Z}^*$. From this equality and the fact that $S$  maximizes $|S|-|N_{B_L^f}(S)|$ over all subsets of $V_f$, we have for any $T\subseteq V_f\setminus S$,
\begin{equation}
|\textup{abs}(N_{B_L^f}(T)\setminus (N_{B_L^f}(S))^{\pm})|\ge |T|.
\end{equation}

By  Theorem \ref{signHall}, there is a good matching $M$ in $B_L^f-(N_{B_L^f}(S))^{\pm}$ that saturates $V_f\setminus S$. Let $A_1$ be the union of all color classes in $V_f\setminus S$ and $g_1$ be the mapping on $A$ corresponding to $M$. As $f$ is proper, one easily see that $g_1$ is an $L_{A_1}$-coloring. Let $A_2$ be the union of all color classes in $S$ with at least two vertices and $g_2$ be $f$ restricted on $A_2$. As $f$ is a weak $L$-coloring, one easily check that $f(v)\in L(v)$ for each $v\in A_2$ since otherwise  $f^{-1}(f(v))=\{v\}$, contradicting the definition of $A_2$. This proves that $g_2$ is an $L_{A_2}$-coloring. Thus $g_2(A_2)\subseteq N_{B_L^f}(S)$. Note that $g_1(A_1)\subseteq C_L\setminus (N_{B_L^f}(S))^{\pm}$. This implies $(g_1(A_1))^{\pm}\cap (g_2(A_2))^{\pm}=\emptyset$, i.e., $|c_1|\neq|c_2|$ for any $c_1\in g_1(A_1)$ and $c_2\in g_2(A_2)$. Thus we can obtain an $L_{A}$ coloring by combining $g_1$ with $g_2$. This proves (a). If $A=V(\Sigma)$ then we are done. We assume that $A$ is a proper subset of $V(\Sigma)$, i.e., $X\neq \emptyset$. For every vertex $v\in X$, we see that $\{v\}$ is a color class of $f$ contained in $S$ and hence $L(v)\subseteq N_{B_L^f}(S)$. Consequently,
$$L(v)\cap (g_1(A_1))^{\pm}\subseteq  (N_{B_L^f}(S))\cap (g_1(A_1))^{\pm}=\emptyset,$$ and hence
$$|L(v)\setminus(g(A))^{\pm}|=|L(v)\setminus(g_2(A_2))^{\pm}|\ge 2(k-|A_2|).$$
As $|A_2|\le \ell$, we have $|L(v)\setminus(g(A))^{\pm}|\ge 2(k-l)$, that is, (d) holds.

As $A$ contains all color classes with more than one element, we have $|A|\ge 2\ell$. Consequently $|V(\Sigma)\setminus A|\le 2(k-\ell)+1$, i.e., (b) holds.
Since $f$ is surjective and $C$ is a representative subset of $C_L$, we have $\ell\le |V(\Sigma)|-|C|=|V(\Sigma)-|\textup{abs}(C_L)|$, i.e., $\ell\le \gamma_L$. Note that $A$ contains each color class in $V_f\setminus S$.   As $V_f\setminus S$ contains at least $\gamma_L$ singletons of $\Sigma$, we know that $\Sigma-A$ has at most $k-\gamma_L$ parts and hence $\chi^*(\Sigma-A)\le 2(k-\gamma_L)$. Thus, $\chi^*(\Sigma-A)\le 2(k-\ell)$, so (c) holds.  This complete the proof of this lemma.
\end{proof}
\begin{prop}\label{nonear}
There is no near $L$-coloring.
\end{prop}
\begin{proof}
Suppose to the contrary that there is a near $L$-coloring $f$.  Let $C=f(V(\Sigma))$. By Proposition \ref{nearCsur}, we may assume that $C$ is a representative subset of $C_L$. If $M$ is a good matching in $B_L^f$ that saturates $V_f$, then $M$ clearly indicates an $L$-coloring with the same color classes as $f$. This is a contradiction to our assumptions on  $\Sigma$ and $L$.  Therefore, no such a matching exists. By Theorem \ref{signHall}, there is a subset $S$ of $V_f$ such that $|S|>|\textup{abs}(N_{B_L^f}(S))|$. We assume further that $|S|-|\textup{abs}(N_{B_L^f}(S))|$ is maximized over all subsets of $V_f$.

Let $S=\{f^{-1}(c_1),f^{-1}(c_2),\ldots,f^{-1}(c_p)\}$. As $\{c_1,c_2,\ldots,c_p\}\subseteq C$, we see that $|c_1|,|c_2|,\ldots,|c_p|$ are pairwise distinct. Since $|S|>|\textup{abs}(N_{B_L^f}(S))|$, at least one of $c_i's$, say $c_1$, does not belong to $N_{B_L^f}(S)$. In particular, $c_1\not \in N_{B_L^f}(f^{-1}(c_1))$. By the construction of the edge set in  $B_L^f$, we know that there is a vertex $v\in f^{-1}(c_1)$ such that $c_1\not\in L(v)$. Of course, $f(v)=c_1$. Since $f$ is a near $L$-coloring, it must happen that $c_1$ is frequent and $f^{-1}(c_1)=\{v\}$.

To obtain a contradiction, we consider two cases:

\noindent\emph{Case} 1.  $c_1$ is globally frequent for $L$.

Since $f^{-1}(c_1)=\{v\}\in S$, we have that $N_{B_L^f}(S)\supseteq N_{B_L^f}(f^{-1}(c_1))=L(v)$ and hence $|N_{B_L^f}(S)|\ge 2k$. This implies that
\begin{equation}\label{Slargerk}
|S|> |\textup{abs}(N_{B_L^f}(S))|\ge k
\end{equation}

On the other hand, since $c_1\not\in N_{B_L^f}(\{f^{-1}(c_1),f^{-2}(c_2),\ldots,f^{-1}(c_p)\})$, each $f^{-1}(c_i)$ must contain a vertex whose list does not contain $c_1$. Since $c_1$ is globally frequent for $L$, the number of such vertices is at most $|V(\Sigma)|-(k+1)$, which is $k$ by Corollary \ref{exactorder}.
Thus $p\le k$, i.e., $|S|\le k$. This contradicts (\ref{Slargerk}).

\noindent\emph{Case} 2. $c_1$ is  frequent among singletons  for $L$.

Let $\{v_1\},\{v_2\},\ldots,\{v_{\gamma_L}\}$ be $\gamma_L$ singletons such that $c_1\in L(v_i)$ for each $i$. As $f$ is proper, each singleton $\{v_i\}$ with $i\in \{1,2,\ldots,\gamma_L\}$ is a color class of $f$, that is, $\{v_i\}\in V_f$. If $\{v_i\}\in S$ for some $i\in \{1,2,\ldots,\gamma_L\}$ then
we have $N_{B_L^f}(S)\supseteq N_{B_L^f}(\{v_i\})=L(v_i)\ni c_1$, a contradiction. Thus  $V_f\setminus S$ contains at least $\gamma_L$ singletons $\{v_1\},\{v_2\},\ldots,\{v_{\gamma_L}\}$. Now, by Lemma \ref{manysingL}, $\Sigma$ is $L$-colorable. This is a contradiction.

In either case, we have a contradiction. The proof is complete.
\end{proof}
\subsection{Upper bound on the number of frequent colors for $L\in \mathcal{L}$}
\begin{lem}\label{basicupper2k}
For each $L\in \mathcal{L}$, $C_L$ contains at most $2(k-1)$ frequent colors.
\end{lem}
\begin{proof}
Suppose to the contrary that there is an $L\in \mathcal{L} $ for which $C_L$ contains at least $2k-1$ frequent colors.  Let $F$ be a set of $k$ frequent colors with different absolute values. We shall construct a near $L$-coloring, which contradicts Proposition \ref{nonear} and hence completes the proof.

We construct a near $L$-coloring by a three-phase greedy procedure. In the first phase, choose a subset $V_1$ and an $L_{V_1}$-coloring $f_1\colon\,V_1\mapsto C_L\setminus F^{\pm}$ such that $V_1$ contains as many vertices as possible, and subject to this, $V_1$ contains vertices from as many parts as possible.

\noindent\textbf{Claim 1.} \emph{ Every part of size 2 contains a vertex of $V_1$.}
\begin{proof}[Proof of the claim]  Suppose to the contrary that there is a part $P=\{u,v\}$ such that $P\cap V_1=\emptyset$. By Corollary \ref{nonsingledisjoint}, we have $L(u)\cap L(v)=\emptyset$. Now Corollary \ref{2disjoint} implies that $L(u)\cup L(v)=C_L$, $|C_L|=4k$ and $|\textup{abs}(C_L)|=2k$. Thus $C_L$ is a \emph{symmetric} set, i.e., $C=-C$. Thus $C_L\setminus F^{\pm}$ is a symmetric set consisting of $k$ pairs of opposite colors. As $f_1(V_1)\subseteq C_L\setminus F^{\pm}$, we have $(f_1(V_1))^{\pm}\subseteq C_L\setminus F^{\pm}$. If $(f_1(V_1))^{\pm}\neq C_L\setminus F^{\pm}$ then we can use a color $c\in C_L\setminus F^{\pm}$ to color $u$ or $v$, increasing the size of $V_1$ and hence contradicting maximality of $|V_1|$. Therefore $(f_1(V_1))^{\pm}= C_L\setminus F^{\pm}$. As $f$ is proper, all colors in $f(V_1)$ must have different absolute values. This implies that $|f_1(V_1)|=k$ and hence $|V_1|\ge k$ with equality holding if and only $f_1$ is injective.

If $|V_1|\ge k+1$ then $|V(\Sigma)\setminus V_1|\le k$ and we can extend $f_1$ to obtain a near $L$-coloring by mapping the vertices in $V(\Sigma)\setminus V_1$ injectively to $F$. This is a contradiction to Proposition \ref{nonear}.

Now consider the case that $|V_1|=k$ and hence $f_1$ is injective. Since neither vertex of $P$ is in $V_1$ and $\Sigma$ has exactly $k$ parts, there must be a part $Q$ containing at least two vertices of $V_1$, say $x$ and $y$. However, since $L(u)\cup L(v)=C_L$, we can uncolor $x$ and use its color to color one of $u$ and $v$. This maintains the number of colored vertices but increases the number of parts with a colored vertex. This contradicts our assumption on $V_1$. Claim 1 follows as we obtain a contradiction in either case.
\end{proof}

For each part $P$, let $R_P=P\setminus V_1$, the set of vertices that are not colored by $f_1$. Label the parts of $\Sigma$ as $P_1$,$P_2$,\ldots,$P_k$ so that $|R_{P_1}|\ge {R_{P_2}}\ge\cdots\ge|R_{P_k}|$. The second phase of our coloring procedure is described as follows. For each part $R_{P_i}$, in turn, we try to color $R_{P_i}$ with a color,  say $c_i\in F^{\pm}$ such that neither $c_i$ nor $-c_i$ has yet been used and $c_i$ is available for all vertices in $R_{P_i}$. We stop when either $i=k$, or arrive at a part $R_{P_{i+1}}$ for which we fail to color $R_{P_{i+1}}$. If $i=k$ then all vertices have been colored and we have obtained an $L$-coloring, a contradiction. Thus, $i<k$. Let $U=F^{\pm}\setminus\{\pm c_1,\pm c_2,\dots,\pm c_i\}$.  We observe that $|U|=2k-2i$ and each color in $U$ is absent from $L(v)$ for at least one $v\in R_{P_{i+1}}$.

Define $V_2=R_{P_1}\cup R_{P_2}\cup\cdots \cup R_{P_i}$ and $V_3=R_{P_{i+1}}\cup R_{P_{i+2}}\cup\cdots\cup R_{P_k}$. Note that $|U\cap F|=k-i$. If $|V_3|\le k-i$ then in the third phase we simply map $V_3$ injectively into $U\cap F$, which gives a near $L$-coloring,  a contradiction. Thus $|V_3|\ge k-i+1$, implying
\begin{equation}\label{RPim1}
|R_{P_{i+1}}|\ge 2
\end{equation} by our choice of ordering. Now, $|V_2|\ge |R_{P_i}|i\ge 2i$. As $(V_1,V_2,V_3)$ is clearly a partition of $V(\Sigma)$ and $|V(\Sigma)|=2k+1$, we have
\begin{equation}\label{V1smaller}
|V_1|=(2k+1)-|V_2|-|V_3|\le (2k+1)-2i-(k-i+1)=k-i.
\end{equation}
Let us show that the inequality in (\ref{V1smaller}) is indeed an equality, that is
\begin{equation}\label{V1kmi}
|V_1|=k-i.
\end{equation}
 Note that $U$ contains $(2k-2i)$ colors from $F^{\pm}$ and each of these color is absent from $L(v)$ for at least one $v\in R_{P_{i+1}}$. Since all lists of $R_{P_{i+1}}$ have size at least $2k$ colors and $|F^{\pm}|=2k$, these absences imply that the colors of $C_L\setminus F^{\pm}$ must appear at least $(2k-2i)$ times (in total) among these lists. For each color $c\in C_L^{\pm}$, we use $n(c)$ to denote the number of lists of $R_{P_{i+1}}$ which contain $c$. Then we have
\begin{equation}\label{addnc}
\sum\limits_ {c\in C_L\setminus F^{\pm}} n(c)\ge 2k-2i,
\end{equation}
or equivalently,
 \begin{equation}\label{addncnmc}
 \sum_{\substack{c>0\\  c\in C_L^{\pm}\setminus F^{\pm}}} n(c)+n(-c)\ge 2k-2i.
 \end{equation}

 Let $c$ be any color with $n(c)>0$. If $f_1^{-1}(c)\cap P_{i+1}\neq\emptyset$ then in the first phase we could use $c$ to color $n(c)$ vertices of $R_{P_{i+1}}$, thereby increasing the size of $V_1$, a contradiction. Thus, we have established the following fact.

\noindent\textbf{Fact 1}. \emph{ If $n(c)>0$ then $f_1^{-1}(c)\cap P_{i+1}=\emptyset$, i.e., $c$ was not used to color any vertex of $P_{i+1}$ in the first phase.}

Let $\{c,-c\}$ be any pair of opposite colors with $n(c)+n(-c)>0$. Without loss of generality, we may assume that $n(c)\ge n(-c)$. If $|f_1^{-1}(\pm c)|<n(c)$ then we could uncolor $f_1^{-1}(\pm c)$ and use $c$ to color $n(c)$ vertices of $R_{P_{i+1}}$, again contradicting the maximality of $|V_1|$. Thus, we have proved

\noindent\textbf{Fact 2}. \emph{For any color $c\in C_L^{\pm}$, $|f_1^{-1}(\pm c)|\ge \max\{n(c),n(-c)\}$, i.e.,  at least $\max\{n(c),n(-c)\}$ vertices are colored with $\pm c$ in the first phase.}

From Fact 2 and Inequality (\ref{addncnmc}), we have

\begin{eqnarray}
|V_1|&\ge &\sum_{\substack{c>0\\  c\in C_L^{\pm}\setminus F^{\pm}}} \max\{n(c),n(-c)\} \label{V1one}\\
& \ge &\sum_{\substack{c>0\\  c\in C_L^{\pm}\setminus F^{\pm}}} \frac{n(c)+n(-c)}{2} \label{V1two}\\
& \ge  & k-i.\label{V1three}
 \end{eqnarray}
This, together with (\ref{V1smaller}), proves that $|V_1|=k-i$. Furthermore, all equalities  must hold in (\ref{V1one})-(\ref{V1three}). Thus,  $n(c)=n(-c)$ for any $c\in C_L^{\pm}\setminus F^{\pm}$, and hence
\begin{equation}\label{V1equal}
|V_1|=\sum_{\substack{c>0,~n(c)>0\\  c\in C_L^{\pm}\setminus F^{\pm}}} n(c).
\end{equation}
For any $c\in C_L^{\pm}\setminus F^{\pm}$ with $n(c)>0$, we have $n(-c)=n(c)>0$ and hence by Fact 1, $f_1^{-1}(\pm c)\cap P_{i+1}=\emptyset$. This, together with Fact 2, leads to
$$
|(V_1\setminus P_{i+1})\cap f^{-1}(\pm c)|\ge n(c)
$$
and hence
\begin{equation}\label{V1mlarger}
|V_1\setminus P_{i+1}|\ge\sum_{\substack{c>0,~n(c)>0\\  c\in C_L^{\pm}\setminus F^{\pm}}} n(c).
\end{equation}
From (\ref{V1equal}) and (\ref{V1mlarger}), we have $V_1\cap P_{i+1}=\emptyset$ and hence $|P_{i+1}|\neq 2$ by Claim 1. Now $R_{P_{i+1}}=P_{i+1}$ and hence $|P_{i+1}|\ge 2$ by (\ref{RPim1}). Thus $|R_{P_{i+1}}|=|P_{i+1}|\ge 3$ and hence $|R_{P_{i}}|\ge 3$ by our choice of ordering. Consequently, $|V_2|\ge 3i$ and a similar reasoning of  (\ref{V1smaller}) leads to $|V_1|\le k-2i$. Thus, by (\ref{V1kmi}), we  have $k-i\le k-2i$ and hence $i=0$.

 Therefore, we have $|V_1|=k$, $|V_2|=0$, $R_{P_1}=P_1$ and $|P_1|\ge 3$.  As we have proved that equality holds in (\ref{V1three}), the equivalence between (\ref{V1three}) and (\ref{addnc}) means
 \begin{equation}\label{addnc2k}
 \sum\limits_ {c\in C_L\setminus F^{\pm}} n(c)= 2k-2i=2k.
 \end{equation}
 As $P_1$ is a nonsingleton part of $\Sigma$, Proposition \ref{nonsingledisjoint} implies that $\cap_{v\in P_1}L(v)=\emptyset$. Note that $R_{P_1}=P_1$. Thus, for any color $c\in F^{-}$, we have $n(c)\le |P_1|-1$ and hence
 \begin{equation}\label{addncFm}
 \sum_{c\in F^{-}}n(c)\le k(|P_1|-1).
 \end{equation}
 As $|L(v)|\ge 2k$ for each $v\in P_1$, we obtain, using double counting,
 \begin{equation}\label{addncCL}
 \sum_{c\in C_L}n(c)=\sum_{v\in P_1}|L(v)|\ge 2k|P_1|.
 \end{equation}
 It follows from (\ref{addnc2k})-(\ref{addncCL}) that
 \begin{eqnarray}
\sum_{c\in F}n(c)
&=&\sum_{c\in F^{\pm}}n(c) -\sum_{c\in F^{-}}n(c) \\
& = &\sum_{c\in C_L}n(c)-\sum_{c\in C_L\setminus F^{\pm}}n(c)-\sum_{c\in F^{-}}n(c)\\
&\ge & 2k|P_1|-2k-k(|P_1|-1)\\
&=&k(|P_1|-1)\\
&\ge &2k.
 \end{eqnarray}
 Therefore, there is a color $c\in F$ with $n(c)\ge 2$. Let $u$ and $v$ be two vertices in $P_1$ such that $c\in L(u)\cap L(v)$. Note that $|V_3|=k+1$ and hence $|V_3\setminus\{u,v\}|=|F\setminus \{c\}|$.  Let $f_2$ be any bijection from $V_3\setminus\{u,v\}$ to $F\setminus \{c\}$. In the third phase, we color $u$ and $v$ with $c$ and map the vertices of $V_1\setminus\{u,v\}$ to $F\setminus \{c\}$ by $f_2$ to give a near $L$-coloring of $\Sigma$. This completes the proof of this lemma.
 \end{proof}
 Now we can obtain a basic inequality between $\Sigma_L$ and $\xi$, where $\Sigma$ is defined in (5), and $\xi$ is the number of singletons in $\Sigma$.
 \begin{cor} \label{xigam}
 $\xi\ge \gamma_L$.
 \end{cor}
 \begin{proof}
 Let $F'$ denote the set of globally frequent colors for $L$. Then each color in $C_L\setminus F'$ appears in at most $k$ lists of $\Sigma$. By Corollary \ref{exactorder}, $|V(\Sigma)|=2k+1$. As there are exactly $k-\xi$ nonsingleton parts, Corollary \ref{nonsingledisjoint} implies that each color $c\in F'$ appears in at most $(2k+1)-(k-\xi)$  lists. Therefore, we have
 \begin{eqnarray}
\sum_{c\in C_L}|N_{B_L}(c)|&=&\sum_{c\in C_L\setminus F'}|N_{B_L}(c)|+\sum_{c\in  F'}|N_{B_L}(c)|\nonumber\\
&\le &k(|C_L|-|F'|)+((2k+1)-(k-\xi))|F'| \nonumber\\
& = &k|C_L|+(\xi+1)|F'|\nonumber
 \end{eqnarray}
 On the other hand, using double counting and the fact that $|L(v)|\ge 2k$, we obtain
 \begin{equation}
 \sum_{c\in C_L}|N_{B_L}(c)|= \sum_{v\in V(\Sigma)}|L(v)|\ge 2k|V(\Sigma)|\nonumber
 \end{equation}
 Combining the above two inequalities leads to
 \begin{equation}\label{Fp}
 |F'|\ge \frac{2k(|V(\Sigma)|-|\textup{abs}(C_L)|)}{\xi+1}=\frac{2k\gamma_L}{\xi+1}.
 \end{equation}
 By Lemma \ref{basicupper2k}, $|F'|\le 2(k-1)<2k$. Thus $\xi+1>\gamma_L$ that is, $\xi\ge \gamma_L$ as $\xi$ is an integer.
 \end{proof}
 The following corollary is immediate from Proposition \ref{xigam} and Definition \ref{deffre}.
 \begin{cor}\label{allsinfre}
 If a color $c\in C_L$ appears in the list of every singleton then $c$ is frequent.
 \end{cor}
 \begin{cor}\label{no2part}
 $\Sigma$ contains no part of size 2.
 \end{cor}
 \begin{proof}
 Suppose to the contrary that $P$ is a part of size 2. Label the two vertices in $P$ as $u$ and $v$. By Corollaries \ref{nonsingledisjoint} and \ref{2disjoint}, $|\textup{abs}(C_L)|=2k$ and hence $\gamma_L=1$. By Corollary \ref{xigam}, $\xi\ge 1$, that is $\Sigma$ contains at least one singleton. Let $v$ be a singleton of $\Sigma$. Each color in $L(v)$ is clearly frequent among singletons for $L$ since $\gamma_L=1$. Thus $L$ has at least $2k$ frequent colors. This contradicts Lemma \ref{basicupper2k} and hence completes the proof.
 \end{proof}
 \begin{cor}\label{xikm2}
 $\xi\ge \frac{k-1}{2}.$
 \end{cor}
 \begin{proof}
 By Corollary \ref{no2part}, $\Sigma$ contains no part of size 2. Note that $G$ has $\xi$ singletons and $k-\xi$ nonsingleton parts.  As $|V(\Sigma)|=2k+1$ by Corollary \ref{exactorder}, we have
 $3(k-\xi)+\xi\le 2k+1$, which implies $\xi\ge \frac{k-1}{2}$.
 \end{proof}

 We say $L\in \mathcal{L}$ is a \emph{maximal} list assignment if for any $v\in V(\Sigma)$ with $C_L\setminus L(v)\neq \emptyset$ and any $c\in C_L\setminus L(v)$, there is an $L^*$-coloring of $\Sigma$, where $L^*$ is defined by
  \begin{equation}\label{inlarge}
  L^*(u)=\begin{cases}
 L(v)\cup\{c\} & \text{if~}u=v,\\
L(u) & \text{if~} u\in V(\Sigma)\setminus\{v\}.
\end{cases}
  \end{equation}
We use $\mathcal{L}_{\max}$ to denote the set of all maximal list assignments. We note that $\mathcal{L}_{\max}$ is a nonempty subset of $\mathcal{L}$.
\begin{lem}\label{keepCLfre}
For each $L\in \mathcal{L}$, there is an $L'\in \mathcal{L}_{\max}$ such that $L(v)\subseteq L'(v)$ and $C_L=C_{L'}$. In particular, each color frequent for $L$ is also frequent for $L'$.
\end{lem}
\begin{proof}
If $L$ is maximal, then we take $L'=L$ and we are done. If $L$ is not maximal then there are $v\in V(\Sigma)$ and $c\in C_L$  such that for the list assignment $L^*$ as defined in (\ref{inlarge}), $\Sigma$ is not $L^*$-colorable. Clearly $L^*(v)\subseteq L(v)$, $C_L=C_{L'}$. Thus $\gamma_L=\gamma_{L^*}$ and each color frequent for $L$ is clearly frequent for $L^*$. If $L^*$ is maximal, we are done. Otherwise, we repeat the above process, which will clearly terminate and hence obtain a maximal list assignment with the desired properties.
\end{proof}
\begin{lem}\label{freevesin}
Let $L\in \mathcal{L}_{\max}$. If $c\in C_L$ is frequent, then $c\in L(v)$ for every singleton $v$ of $\Sigma$.
\end{lem}
\begin{proof}
Suppose to the contrary that there is a singleton $v$ such that $c\not\in L(v)$. Let $L^*$ be defined as in (\ref{inlarge}).  Since $L$ is maximal, we see that $\Sigma$ is $L^*$-colorable. Let $f$ be an $L^*$-coloring. As $\Sigma$ is not $L$-colorable, we must have $f(v)=c$. Moreover, since $v$ is a singleton and $f$ is proper, we have that $f^{-1}(f(v))=\{v\}$. Now it is easy to see that $f$ is a near $L$-coloring as  $c$ is frequent for $L$. This contradicts Proposition \ref{nonear} and hence completes the proof.
\end{proof}
In Lemma \ref{basicupper2k}, we show that for any $L\in \mathcal{L}$ the number of frequent colors with different absolute values is less than $k$. The following result gives a better upper bound, which is the key result of this section.
\begin{prop}\label{kmxi}
For any $L\in \mathcal{L}$, there are at most $2(k-\xi-1)$ frequent colors.
\end{prop}
\begin{proof}
Suppose to the contrary there is an $L\in \mathcal{L}$ for which $C_L$ has at least $2(k-\xi)-1$ frequent colors. By Lemma \ref{keepCLfre}, we may assume $L$ is maximal. Let $F_0=\{c_1,c_2,\ldots,c_{k-\xi}\}$ be a set of $(k-\xi)$ frequent colors  with different absolute values. Label the singletons as $v_1,v_2,\ldots,v_{\xi}$. By Lemma \ref{freevesin}, we have $F_0\subseteq L(v_i)$ for $i\in \{1,2,\ldots,\xi\}$. For each $i\in \{1,2,\ldots,\xi\}$, in turn, we choose a color $c_{k-\xi+i}\in L(v_i)\setminus F_{i-1}^{\pm}$ and define $F_i=F_{i-1}\cup\{c_{k-\xi+i}\}$. We note that $|F_{\xi}|=k$ and $|F_{\xi}^{\pm}|=2k$. Let $L'$ be the list assignment defined by
  \begin{equation}
  L'(v)=\begin{cases}
 F_{\xi}^{\pm} & \text{if~ $v=v_i$ for some $i\in \{1,2,\ldots,\xi\}$} ,\\
 L(v) & \text{otherwise}.
\end{cases}
  \end{equation}
\noindent\textbf{Claim 1.} \emph{There is an $L'$-coloring of $\Sigma$.}
\begin{proof}[Proof of Claim 1]
Suppose to the contrary that $\Sigma$ is not $L'$-colorable. As $|L'(v)|\ge 2k$ for all $v\in V(\Sigma)$, we have $L'\in \mathcal{L}$. By Corollary \ref{allsinfre}, all colors in $F_{\xi}^{\pm}$ is frequent. Thus, there are at least $2k$ frequent colors for $L'$. This is a contradiction to Lemma \ref{basicupper2k} and hence Claim 1 holds.
\end{proof}
Let $f'$ be an $L'$-coloring of $\Sigma$. Let $S$ be $\{v_1,v_2,\ldots,v_{\xi}\}$,  the set of all singletons. Clearly, $f'(S)=\{f'(v_1)$, $f'(v_2),\ldots, f'(v_{\xi})\}$ consists of $\xi$ colors in $F_{\xi}^{\pm}$ with different absolute values. Moreover, if $f'(v_i)\in F_{\xi}^{-1}$ for some $i\in \{1,2,\ldots,\xi\}$ then we can opposite the color of $v_i$ and the resulting mapping is also an $L'$-coloring. Therefore, we may assume that $f'(S)\subseteq F_{\xi}$. Let $$S'=\{v_i\colon\,1\le i\le \xi \text{~and~} c_{k-\xi+i}\in f'(S)\}.$$  Thus  $|S'|=|f'(S)\cap\{c_{k-\xi+i}\colon\,1\le i\le \xi\}|$. As $f'(S)\subseteq F_{\xi}$ and $F_{\xi}=F_0\cup \{c_{k-\xi+i}\colon\,1\le i\le \xi\}$, we have
$$|S|=|f'(S)|=|f'(S)\cap F_0|+|f'(S)\cap \{c_{k-\xi+i}\colon\,1\le i\le \xi\}|=|f'(S)\cap F_0|+|S'|$$
and hence $|S\setminus S'|=|S|-|S'|=|f'(S)\cap F_0|$. Let $f''$ be an arbitrary bijection from $S\setminus S'$ to $f'(S)\cap F_0$ and define a mapping $f$ on $V(\Sigma)$ as follows:
\begin{equation}\label{ffpfpp}
f(v)=\begin{cases}
 f''(v) & \text{if~ $v\in S\setminus S'$} ,\\
 f'(v) & \text{otherwise}.
\end{cases}
  \end{equation}
As $F_0\subseteq L(v)$ for all $v\in S$, one easily finds that $f(v)\in L(v)$ for all $v\in V(G)$. Clearly,
$$f(S)=f''(S\setminus S')\cup f'(S')=(f'(S)\cap F_0)\cup (f'(S)\cap \{c_1,c_2,\ldots,c_{\xi}\})=f'(S)$$
That is, $f$ is obtained from $f$ by permutating the colors of all singletons, which implies that $f$ is proper. Thus $f$ is an $L$-coloring of $\Sigma$. This is a contradiction and hence completes the proof.
\end{proof}
Using Proposition \ref{kmxi}, we can improve Corollary \ref{xigam} as follows.
\begin{cor}\label{xi2gamma}
$\xi\ge 2\gamma_L$
\end{cor}
\begin{proof}
Let $F'$ be the set of globally frequent colors. We use the same argument as in the proof Corollary \ref{xigam}. By (\ref{Fp}) and Propositions \ref{kmxi} and \ref{xikm2}, we have
\begin{equation}
\frac{2k\gamma_L}{\xi+1}\le |F'|\le 2(k-\xi-1)\le k-1<k,
\end{equation}
and hence $\xi+1> 2\gamma_L$. Thus, $\xi\ge 2\gamma_L$ as $\xi$ is an integer.
\end{proof}
\subsection{Lower bound on the number of frequent colors for $L\in \mathcal{L_{\max}}$ }
The main aim of this section is to show that for any maximal list assignment $L\in\mathcal{L}_{\max}$, there are at least $2k-2\xi+1$ frequent colors. This contradicts Proposition \ref{kmxi} and finally completes the whole proof of Theorem \ref{equmain}. In the following we assume $L$ is a maximal list assignment.

\begin{lem}\label{manyVfewC}
If a color $c^*\in C_L$ does not appear in the list of a singleton $v$, then there is a nonempty set $X=X(c^*)$ of singletons such that

 \textup{(a)} $|X|\ge \xi-\gamma_L+1$, and

 \textup{(b)} $|\textup{abs}(\cup_{x\in X} L(x))|\le 2k-|N_{B_L}(c^*)|$.
\end{lem}
\begin{proof}
Let $L^*$ be defined as in (\ref{inlarge}). As $L$ is maximal, there is an $L^*$-coloring $f_1$ of $\Sigma$. Since $\Sigma$ is not $L$-colorable, $f_1$ must use $c$ to color $v$. Moreover, as $v$ is a singleton, $v$ is the only vertex with color $c^*$ under $f_1$. Note that for any vertex $u$ in $\Sigma$ different from $v$, $L^*(u)=L(u)$ and hence $f_1(u)\in L(u)$. This indicates that $f_1$ is a weak $L$-coloring. By Proposition \ref{nearCsur}, there are a representative subset $C$ of $C_L$ and a weak $L$-coloring $f\colon\,V(\Sigma)\mapsto C$ such that for any $u\in V(\Sigma)$, (\RNum{1}) $f_1(u)\in L(u)$ implies $f(u)\in L(u)$  and (\RNum{2}) $f(u)\not\in L(u)$ implies $f(u)=f_1(u)$. Therefore, for any vertex $u\in V(\Sigma)\setminus \{v\}$, $f(u)\in L(u)$ as $f_1(u)\in L(u)$. Moreover, as $\Sigma$ is not $L$-colorable, $f(v)\not\in L(v)$ and hence $f(v)= f_1(v)=c^*$. We also note that $v$ is the only vertex with color $c^*$ under $f$ as $v$ is a singleton.

If there is a good matching that saturates $V_f$ then $\Sigma$ is $L$-colorable, a contradiction. Thus, by Theorem \ref{signHall}, there is a set $S\subseteq V_f$ such that $|\textup{abs}(N_{B_L^f}(S))|<|S|$. We assume further that $S$ maximizes $|S|-|\textup{abs}(N_{B_L^f}(S))|$.

Note that for any $u\in V(\Sigma)\setminus \{v\}$, we have $f(u)\neq c^*$ and $f(u)\in L(u)$. Thus, for any $c\in f(V(\Sigma))\setminus\{c^*\}$, we have $c\in \cap_{u\in f^{-1}(c)} L(u)$, that is $c$ is adjacent to $f^{-1}(c)$ in $B_L^f$. Since $|\textup{abs}(N_{B_L^f}(S))|<|S|$ and colors in $f(V(\Sigma))$ have different absolute values, we must have $f^{-1}(c^*)\in S$ and $c^*\not\in N_{B_L^f}(S)$.  Therefore every color class of $S$ must contain a vertex whose list does not contain $c^*$. It follows that
\begin{equation}\label{boundlists}
|\textup{abs}(N_{B_L^f}(S))|<|S|\le |V(\Sigma)|-|N_{B_L}(c^*)|=2k+1-|N_{B_L}(c^*)|.
\end{equation}
Now define $X$ to be the set of all singletons of $\Sigma$ whose color classes under $f$ belong to $S$. As $f^{-1}(c^*)=\{v\}$ and $f^{-1}(c^*)\in S$, we know that $v\in X$ and hence $X$ is nonempty. Clearly, $\cup_{x\in X} L(x)\subseteq N_{B_L^f}(S)$ and hence (b) holds by  (\ref{boundlists}).

Finally, since $\Sigma$ is not $L$-colorable, Lemma \ref{manysingL} implies that $V_f\setminus S$ contains fewer than $\gamma_L$ singletons, i.e., $S$ contains more than $\xi-\gamma_L$ singletons. This proves (a).
\end{proof}
We let $c^*$ be a color that is not frequent, and subject to this, maximizes $|N_{B_L}(c^*)|$. We note that $c^*$ exists since not all colors in $C_L$ are frequent by Lemma \ref{basicupper2k}. Moreover, by Corollary \ref{allsinfre}, there is a singleton $v$ such that $c^*\not\in L(v)$.  Let $X=X(c^*)$ be a set of singletons as described in Lemma \ref{manyVfewC}. Let $p=|N_{B_L}(X)|$ and label the colors in $N_{B_L}(X)$ as $c_1,c_2,\ldots,c_p$ such that the $p$-term sequence $\{|N_{B_L}(c_i)\cap X|\}$ is decreasing.
\begin{defi}\label{defbeta}
$\beta=k-|N_{B_L}(c^*)|.$
\end{defi}
We note that $\beta\ge 0$ as $c^*$ is not frequent.
\begin{prop}\label{ifbetasmall}
If $\beta\le 2(\xi-2\gamma_L+1)$ then $|N_{B_L}(c_{2k-2\xi+1})\cap X|\ge \gamma_L$.
\end{prop}
\begin{proof}
Let $Z=\{c_1,c_2,\ldots,c_{2k-2\xi}\}$ and $Y=N_{B_L}(X)\setminus Z$. Note that $N_{B_L}(X)=\cup_{x\in X} L(x)$.  For each $x\in X$ we have $|L(x)\cap Y|\ge |L(x)|-|Z|\ge 2k-(2k-2\xi)\ge 2\xi$. Thus, by double counting and our choice of ordering,
  \begin{eqnarray}
|Y||N_{B_L}(c_{2k-2\xi+1})\cap X|&\ge&\sum_{c\in Y}|N_{B_L}(c)\cap X|\nonumber\\
&= &\sum_{x\in X}|L(x)\cap Y| \nonumber\\
& \ge& 2\xi|X|\label{YNBL}
 \end{eqnarray}
By Proposition \ref{manyVfewC}(b) and definitions of $\beta$ and $Y$,
\begin{equation}\label{Y}
|Y|=|N_{B_L}(X)|-(2k-2\xi)\le (4k-2|N_{B_L}(c^*)|)-(2k-2\xi)=2\beta+2\xi
\end{equation}
From (\ref{YNBL}), (\ref{Y}), using Proposition \ref{manyVfewC}(a), we have
  \begin{equation}\label{NBLcapX}
|N_{B_L}(c_{2k-2\xi})\cap X|\ge\frac{2\xi|X|}{|Y|}\ge \frac{2\xi(\xi-\gamma_L+1)}{2\beta+2\xi}=\frac{\xi(\xi-\gamma_L+1)}{\beta+\xi}
 \end{equation}
Finally, by Corollary \ref{xi2gamma} and the assumption of this proposition, we have
$$\beta\gamma_L+\xi\gamma_L\le 2(\xi-2\gamma_L+1)\frac{\xi}{2}+\xi\gamma_L=\xi(\xi-\gamma_L+1).$$
Combining this with (\ref{NBLcapX}) leads to $|N_{B_L}(c_{2k-2\xi})\cap X|\ge \gamma_L$.
\end{proof}
\begin{prop}\label{betasmall}
$\beta<\frac{1}{2}(\xi-2\gamma_L+1).$
\end{prop}
\begin{proof}
Let $F\subseteq C_L$ be the set of all frequent colors. By Proposition \ref{kmxi}, we have
\begin{equation}\label{F}
|F|\le 2 (k-\xi-1).
\end{equation}
Since $c^*$ maximizes $N_{B_L}(c^*)$ over all colors in $C_L\setminus F$, we have
\begin{equation}\label{CmF}
\sum_{c\in C_L\setminus F}|N_{B_L}(c)|\le |C_L\setminus F||N_{B_L}(c^*)|.
\end{equation}
Note that there are exactly $k-\xi$ nonsingleton parts in $\Sigma$ and each nonsingleton part contains no common color in their lists by Proposition \ref{nonsingledisjoint}. Thus,  each $c\in C_L$ appears in at most $|V(\Sigma)|-(k-\xi)$ lists, i.e.,
\begin{equation} \label{CF}
|N_{B_L}(c)|\le |V(\Sigma)|-(k-\xi)=k+\xi+1
\end{equation}
as $|V(\Sigma)|=2k+1$ by Corollary \ref{exactorder}.
It follows from (\ref{F})-(\ref{CF}) and Definitions of $\beta$ and $\gamma_L$ that
\begin{eqnarray}
\sum_{c\in C_L}|N_{B_L}(c)|&\le& |C_L\setminus F|(k-\beta)+|F|(k+\xi+1)\nonumber \\
&= & (\beta+\xi+1)|F|+(k-\beta)|C_L| \nonumber \\
&\le &2(\beta+\xi+1)(k-\xi-1)+2(k-\beta)|\textup{abs}(C_L)| \nonumber \\
&=&2(\beta+\xi+1)(k-\xi-1)+2(k-\beta)(2k+1-\gamma_L)\nonumber \\
&=&-2(k+\xi-\gamma_L+2)\beta+(\xi+1)(2k-2\xi-2)\nonumber \\
&&+(2k+1-\gamma_L)2k  \label{CNBLc}
\end{eqnarray}
On the other hand, using double counting,
$$\sum_{c\in C_L}|N_{B_L}(c)|=\sum_{v\in V(\Sigma)}|L(v)|\ge (2k+1)2k,$$
which, together with (\ref{CNBLc}), implies
$$(k+\xi-\gamma_L+2)\beta\le (\xi+1)(k-\xi-1)-k\gamma_L.$$
Moreover, by Corollaries \ref{xigam} and \ref{xikm2}, $k+\xi-\gamma_L+2\ge k+2$ and $k-\xi-1\le (k-1)/2<k/2$. Therefore,
$$\beta< \frac{\frac{k}{2}(\xi+1)-k\gamma_L}{k+2}<\frac{1}{2}(\xi+1-2\gamma_L).$$
\end{proof}

By Propositions \ref{betasmall} and \ref{ifbetasmall}, we obtain $|N_{B_L}(c_{2k-2\xi+1})\cap X|\ge \gamma_L$ and hence
$$|N_{B_L}(c_{i})\cap X|\ge \gamma_L,\text{~for~} i\in\{1,2,\ldots,2k-2\xi+1\}$$
by our ordering on colors in $\cup_{x\in X}L(x)$. Thus, we have found $2k-2\xi+1$ colors which are frequent among singletons.  This contradicts Proposition \ref{kmxi} and hence completes the proof of Theorem \ref{equmain}
\section{Discussion}
For technical reasons, we have always forbidden $0$ as a color above. However, it seems  natural to allow $0$ as a color. For $n\ge 1$, define a $n$-set $M_n=\{\pm 1,\pm 2,\ldots,\pm n/2\}$ if $n$ is even, and $M_n=\{0,\pm 1,\pm 2\ldots,\pm (n-1)/2\}$ if $n$ is odd.  M\'{a}\v{c}ajov\'{a} et al.~\cite{Raspaud2016} defined a (proper) $n$-coloring of a signed graph $\Sigma$ to be a mapping $f\colon\,V(\Sigma)\mapsto M_n$ such that for each edge $e$,
\begin{equation}
f(u)\neq \sigma(e)f(v) ~\text{if $e$ connects $u$ and $v$}.
\end{equation}
We say $\Sigma$ is $n$-colorable if $\Sigma$ admits an $n$-coloring. The \emph{chromatic number} $\chi(\Sigma)$, defined in  \cite{Raspaud2016}, is the minimum positive integer $n$ such that $\Sigma$ is $n$-colorable.  The \emph{list chromatic number} \cite{Jin2016,Schweser2017}, denoted $\chi_l(\Sigma)$, is the minimum $n$ such that for any list assignment $L$ with $L(v)\subseteq \mathbb{Z}$ and $|L(v)|\ge n$, there is an $L$-coloring of $\Sigma$. We say $\Sigma$ is \emph{chromatic-choosable} if $\chi_l(\Sigma)=\chi(\Sigma)$.
We believe that the following variant of Theorem \ref{main} may be true.
\begin{conj}\label{future}
If $|V(\Sigma)|\le \chi(\Sigma)+1$, then $\Sigma$ is chromatic-choosable.
\end{conj}
We have shown that Theorem \ref{main} indeed strengthens Noel-Reed-Wu Theorem using the notion of complete extension. One may wonder whether Noel-Reed-Wu Theorem has a direct generalization to a signed graph whose underling graph is simple.
\begin{prob}\label{simpleextension}
 Let $\Sigma$ be a signed graph whose underlying graph is simple. Does $|V(\Sigma)|\le 2\chi(\Sigma)+1$ imply that $\Sigma$ is chromatic-choosable.
 \end{prob}
We give a negative answer to  Problem \ref{simpleextension}. Let $\Sigma$ be a negative complete graph with $4$ vertices. Clearly $\chi(\Sigma)=2$ and hence $|V(\Sigma)|\le 2\chi(\Sigma)+1$ is satisfied. Label the vertices of $\Sigma$ as $v_1,v_2,v_3,v_4$. Define a list assignment $L$ as follows:
 $$L(v_1)=\{1,2\},L(v_2)=\{1,-2\},L(v_3)=\{-1,2\},L(v_4)=\{-1,-2\}.$$
 One easily checks that $\Sigma$ is not $L$-colorable. Thus $\chi_l(\Sigma)>2$ and hence $\Sigma$ is not chromatic-choosable.



\begin{thebibliography}{99}\addtolength{\itemsep}{-1ex}
\bibitem{Ohba2002} K. Ohba, On Chromatic-choosable graphs, J. Graph Theory 40(2)(2002) 130-135.
\bibitem{Noel2015} J. A. Noel, B. A. Reed, H. Wu, A proof of a conjecture of Ohba, J. Graph Theory 79(2)(2015) 86-102.
\bibitem{Jin2016}L. Jin, Y. Kang, E. Steffen, Choosability in signed planar graphs, Europ. J. Combin.  52(2016)234-243.
\bibitem{Raspaud2016} E. M\'{a}\v{c}ajov\'{a}, A. Raspaud, M. \v{S}koviera, The chromatic number of a signed graph, Electro. J. Combin. 23(2016)\#P1.14.
\bibitem{Schweser2017}T. Schweser, M. Stiebitz, Degree choosable signed graphs, Discrete Math. 340 2017 882-891.
\bibitem{Zaslavsky1982} T. Zaslavsky, Signed graph coloring, Discrete Math. 39 (1982) 215-228
\end{thebibliography}
\end{document}